\documentclass[11pt]{amsart}

\usepackage[latin1]{inputenc}
\usepackage{amsmath, amsthm, amssymb, comment}
\usepackage[dvips]{graphicx}
\usepackage{overpic}
\usepackage{enumitem}
\usepackage[colorlinks=true, citecolor=red]{hyperref}
\usepackage{pinlabel}

\usepackage{thm-restate}

\usepackage[margin=3.5cm]{geometry}

\usepackage{color}

\newtheorem{theorem}{Theorem}[section]
\newtheorem{corollary}[theorem]{Corollary}
\newtheorem{proposition}[theorem]{Proposition}
\newtheorem{definition}[theorem]{Definition}
\newtheorem{lemma}[theorem]{Lemma}

\newtheorem*{theorem*}{Theorem}
\newtheorem*{proposition*}{Proposition}
\newtheorem*{definition*}{Definition}
\newtheorem*{lemma*}{Lemma}
\newtheorem*{claim*}{Claim}
\newtheorem*{corollary*}{Corollary}
\newtheorem*{convention*}{Convention}

\theoremstyle{definition}

\newtheorem*{notation}{Notation}
\newtheorem{construction}[theorem]{Construction}

\theoremstyle{remark}
\newtheorem{rem}[theorem]{Remark}
\newtheorem*{rem*}{Remark}
\newtheorem*{acknowledgement}{Acknowledgement}

\newcommand{\wt}[1]{\widetilde{#1}}

\newcommand\bR{\mathbb R}

\newcommand\bZ{\mathbb Z}

\newcommand{\R}{\mathbb R}

\newcommand\eps{\varepsilon}

\newcommand\orb{ \mathcal O }

\newcommand{\cO}{\orb}

\newcommand{\cF}{\mathcal{F}}

\newcommand{\Homeo}{\mathrm{Homeo}}

\newcommand{\diam}{\mathrm{diam}}
\newcommand{\hol}{\mathrm{hol}}

\newcounter{notes}

\author[S. Fenley]{Sergio Fenley}
\address{Department of Mathematics, Florida State University, Tallahassee, FL 32306, USA}
\urladdr{http://www.math.fsu.edu/~fenley}
\email{sfenley@fsu.edu}

\author[K. Mann]{Kathryn Mann}
\address{CNRS, Institut de Math\'ematiques de Jussieu - Paris Rive Gauche}
\urladdr{https://webusers.imj-prg.fr/~kathryn.mann/}
\email{mann@imj-prg.fr}

\author[R. Potrie]{Rafael Potrie}
\address{CMAT, Facultad de Ciencias, Universidad de la Rep\'ublica, Uruguay and IRL2030 'Laboratorio del Plata' CNRS}
\urladdr{www.cmat.edu.uy/$\sim$rpotrie}
\email{rpotrie@cmat.edu.uy}

\title[Exotic codimension one Anosov flows]{Exotic codimension one Anosov flows}

\begin{document}

\begin{abstract}
We construct Anosov flows in certain circle bundles over closed hyperbolic 3-manifolds, producing counterexamples to a conjecture of Verjovsky. Some of these 4-manifolds admit infinitely many distinct Anosov flows up to orbit equivalence. The construction is made by using Cannon-Thurston maps associated to pseudo-Anosov quasigeodesic flows in hyperbolic $3$-manifolds. 
\bigskip


%
\end{abstract}

\maketitle

\section{Introduction} 

A $C^1$ vector field $X$ generating a flow $\phi^t$ on a closed Riemannian manifold $N$ is \emph{Anosov}
if the tangent bundle has a $D\phi^t$-invariant splitting $TN= E^s \oplus \mathbb{R} X \oplus E^u$ with the property that there are constants $c,\lambda>0$ such that
$$ \|D\phi^t v \| \leq c e^{-\lambda t} \|v\| \quad \forall t\geq 0, v \in E^s, \text{ and }$$
$$ \|D\phi^t v \| \leq c e^{\lambda t} \|v\| \quad \forall t\leq 0, v \in E^u.$$

The classification of Anosov systems up to orbit equivalence\footnote{Two flows on $M$ are  orbit equivalent if they differ only by time-change and conjugacy. }
 has been a central problem in dynamical systems at least since Smale's 1966 ICM address  \cite{Smale}. In 1970 Verjovsky conjectured\footnote{In \cite{Verjovsky} this was posed as a problem, but since then circulated as 
 a conjecture.} that if $\mathrm{dim}(M)\geq 4$ and $\phi^t$ is a codimension one Anosov flow (that is, so that either $E^s$ or $E^u$ is one-dimensional) then $\phi^t$ must be orbit equivalent to a suspension of a linear toral automorphism.  See also \cite{Ghys,Verjovsky2}. 

Ghys \cite{Ghys} proved Verjovsky's conjecture under the additional assumption that the flow is volume preserving and the codimension one subbundle is of class $C^2$. This result was improved by Simic  \cite{Simic} who replaced the $C^2$ hypothesis with $C^{1+\alpha}$ for every $\alpha \in (0,1)$. We refer the reader to \cite{Asaoka, Barbot, PlanteThurston, Verjovsky1} and the brief survey in \cite[\S 8.4]{FisherHasselblatt} for more information on the construction and classification of codimension-1 Anosov flows. Since the late 1960's, it has also been known that all codimension 1 Anosov {\em diffeomorphisms} are conjugate to  linear toral automorphisms, due to work of Franks and Newhouse \cite{Franks, Newhouse}.   Remarkably, the situation is quite different for Anosov flows in dimension 3, where non-suspension examples abound.  See \cite{FisherHasselblatt} or \cite{BM} for a general introduction to this topic.

In this article, we construct infinitely many counterexamples to  Verjovsky's conjecture: 

\begin{theorem} \label{thm_counterexample}
There exist codimension one Anosov flows on compact 4-manifolds which are not orbit equivalent to suspension flows.  
These occur on infinitely many distinct manifolds (even up to taking finite covers). Additionally, there exist some compact 4-manifolds which support infinitely many codimension one Anosov flows, all distinct up to orbit equivalence.   
\end{theorem}

All previously known examples of non-suspension Anosov flows had $\dim E^s = \dim E^u$, and this was conjectured to always be true (see \cite[Question 1.10]{BBGH} and \cite[Conjecture 8.4.18]{FisherHasselblatt}).  Theorem \ref{thm_counterexample} disproves this conjecture as well.

To describe the examples in more detail, we recall the standard construction of a foliated bundle. 
For a manifold $M$ with universal cover $\wt M$, and action $\rho$ of $\pi_1(M)$ on the circle $S^1$ by homeomorphisms, we define
$N_\rho : = (\wt M \times S^1) /\pi_1(M)$ 
where the action of $\pi_1(M)$ is 
diagonal, with the action on the first factor by deck transformations on $\wt M$ and on $S^1$ by $\rho$.  The manifold $N_\rho$ is topologically an $S^1$-bundle over $M$, with a codimension one foliation $\cF^h$ transverse to the $S^1$ fibers, induced from the horizontal foliation of $\wt M \times S^1$.  

The main ingredient in Theorem \ref{thm_counterexample} is the following: 

\begin{theorem} \label{thm_main} 
Let $M$ be a closed hyperbolic $n$-manifold. Suppose $\rho: \pi_1(M) \to \Homeo(S^1)$ is a minimal action 
and $f: S^1 \to \partial \wt M$ a continuous map satisfying $f \circ \rho(\gamma) = \gamma \circ f$ for all $\gamma \in \pi_1(M)$. 
Then  $N_\rho$ admits a codimension one topological Anosov flow which is not a suspension, with weak stable foliation $\cF^s$ equal to the horizontal foliation.  

If $f$ has the property that the image of non-injective points have zero (Lebesgue)\footnote{Since $\wt M= \mathbb{H}^n$, one can equivalently take any visual measure on $\partial \mathbb{H}^n\cong S^{n-1}$, which are all in the Lebesgue measure class.} measure in $\partial \wt M$, then $N_\rho$ admits a smooth Anosov vector field, whose flow is orbit equivalent to the topological one above.    
\end{theorem}

A {\em topological Anosov flow} is a flow which admits two invariant, topologically transverse $C^0$-foliations on which nearby orbits have asymptotic behavior paralleling the forwards-convergence / backwards-separation (and vice versa) on weak-stable (resp. weak-unstable) leaves -- see Definition \ref{def_TAF}. 

The outline of the proof of Theorem \ref{thm_main} is as follows:  
Using that the horizontal foliation on $N_\rho$ has hyperbolic leaves, and that $f$ is $\pi_1(M)$-equivariant, we construct a  flow in $N_\rho$ whose orbits along each horizontal leaf are by hyperbolic geodesics with a common ideal point, and which admits a transverse 2-dimensional foliation $\cF^v$.  

From here, the topological and smooth arguments can essentially be read independently.  In the topological case, we follow a strategy used by Calegari \cite{Calegari} in dimension 3 to show orbits separate along leaves of $\cF^v$, and conclude the flow is topologically Anosov with $\cF^v$ the weak-unstable foliation.  This flow is in fact smooth along weak-stable (horizontal) leaves, but only transversely $C^0$.  To improve regularity, we use the additional hypothesis on the map $f$ to apply a trick from Tholozan \cite{Tholozan} which allows us to implement a strategy similar to that of Cawley \cite{Caw} (see also \cite{Asaoka}) to first change the smooth structure on $N_\rho$, then perturb the flow, keeping it tangent to $\cF^h$, to obtain a smooth Anosov flow with respect to this structure. 

To prove Theorem \ref{thm_counterexample}, we need the existence of actual maps $f$ and actions $\rho$ as in the hypothesis of Theorem \ref{thm_main}. Generalizing a construction of Cannon and Thurston \cite{CT}, the first author showed that, for any quasigeodesic pseudo-Anosov flow on a closed hyperbolic 3-manifold, there is an action of $\pi_1(M)$ on a circle and a $\pi_1(M)$-equivariant homeomorphism  $f: S^1 \to \partial \wt M$ (see \cite{Fenley-extension, Fenley-ext2,FenleyQGAnosov}).  
We show here that all of these satisfy the condition that the non-injective points have zero measure, and hence conclude: 

\begin{corollary}\label{coro-4manifold}
If $M$ is a hyperbolic 3-manifold supporting a quasigeodesic pseudo-Anosov flow\footnote{Thanks to \cite{FrankelLandry} it is enough to support a quasigeodesic flow in order to support a quasigeodesic pseudo-Anosov flow.}, then there is a 4-manifold $N_\rho$ as in Theorem \ref{thm_main} admitting a smooth codimension one Anosov flow that is not orbit equivalent to a suspension.  
\end{corollary} 

The reason we take $\dim(M) = 3$ (and hence produce exotic Anosov flows on 4-manifolds only) is that all known examples of maps $f$ and actions $\rho$ satisfying the conditions of Theorem \ref{thm_main} are for 2-dimensional and 3-dimensional manifolds, and the construction in the 2-dimensional case simply recovers lifts of geodesic flow on surfaces to fiberwise covers of the unit tangent bundle.

By using Anosov, rather than pseudo-Anosov examples, which do occur on many hyperbolic 3-manifolds \cite{Fenley-Anosov,BoM,BI-flows,BY}, we can produce infinite families supported on the same 4-manifold:  

\begin{corollary}[Failure of finiteness] \label{cor_infiniteness}
Let $M$ be a hyperbolic 3-manifold admitting a transversally orientable
Anosov flow.  Then, 
there exist infinitely many smooth Anosov flows on $M \times S^1$, pairwise distinct up to orbit equivalence. 
\end{corollary}
To our knowledge, these are the first examples of compact manifolds supporting infinitely many orbit equivalence classes of Anosov flows.

A $\pi_1(M)$ equivariant map $f: S^1 \to \partial \wt{M}$, for some action $\rho$ of $\pi_1(M)$ on $S^1$, is what is typically called a {\em Cannon-Thurston map}.  (This term has also been applied more broadly to various classes of equivariant maps, when at least one of the spaces is the boundary of a group.)  It is an interesting question to classify examples and describe their general properties -- see e.g. \cite{CL}.  It was recently shown that a different class of generalized Cannon-Thurston maps are necessarily finite to one \cite{BHLM}.  Our construction of topological Anosov flows gives the following new result for maps from $S^1$: 

\begin{corollary}[Cannon-Thurston maps are uniformly finite-to-one]\label{cor-CTmaps}
Let $M$ be a closed, negatively curved 
$n$-manifold 
and $\rho$ a minimal action of $\pi_1(M)$ on $S^1$.  If there exists a $\pi_1$-equivariant map $f:  S^1 \to \partial \wt{M}$, then there exists $K$ such that for any $z \in \partial \wt{M}$, the pre-image $f^{-1}(z)$ contains at most $K$ points.
\end{corollary} 

\begin{acknowledgement}
This work was done while the authors were in residence at the Simons Laufer Mathematical Sciences Institute in Spring 2026 for the program ``Topological and Geometric Structures in low dimensions"
supported by the NSF grant No. DMS-2424139. They are 
grateful for the opportunity, working conditions, and all the colleagues in the program for the interesting discussions from which this work benefited.
SF was also partially supported by Simons grant SFI-MPS-TSM-00013757. 
KM was also partially supported by NSF grant DMS-2505228.  RP was also partially supported by CSIC.
We also thank Masayuki Asaoka for comments and questions on an earlier version of this article.
\end{acknowledgement}

\section{Construction of (non-suspension) flows} \label{sec_construction}
While we stated Theorem \ref{thm_main} for hyperbolic manifolds, the construction only requires negative curvature so we start by working in this level of generality.  This is also relevant to obtain Corollary \ref{cor-CTmaps} at the end of this section. 

Fix a closed negatively curved $n$-dimensional manifold $M$. Then $\partial \wt M$ is 
homeomorphic to $S^{n-1}$, and oriented geodesics in $\wt M$ are in one to one correspondence
with pairs of distinct points in $S^{n-1}$.

Let $\rho: \pi_1(M) \to \Homeo(S^1)$ be a minimal action and $f: S^1 \to \partial \wt M$ a $\pi_1(M)$-equivariant continuous map.   Recall that $N_\rho$ denotes the quotient of $\wt M \times S^1$ by $(x, s) \sim (\gamma x, \rho(\gamma)(s))$ where the action on the first factor is by deck transformations.  

We first define some relevant foliations in $\wt M \times S^1$ and later explain that they
project to $N_\rho$.
Let $\cF^h$ denote the horizontal foliation of $\wt{M} \times S^1$ whose leaves are $\wt{M} \times \{p\}$. 
Define a 1-dimensional, oriented, foliation $\mathcal{G}$ on $\wt{M} \times S^1$ by foliating each leaf $\wt{M} \times \{p\}$ of $\cF^h$ by the oriented geodesics of $\wt{M}$ with forward endpoint equal to $f(p)$.  
Since $f$ is continuous, $\mathcal{G}$ is a topological foliation of $\wt{M} \times S^1$. 

\begin{rem} \label{rem_surface}
As an illustrative (but not representative) example, consider the case where $n=2$, the action $\rho$ is the standard action of $\pi_1(M)$ on the boundary circle of $\wt M = \mathbb{H}^2$, and $f: S^1 \to S^1$ is the identity map.  Then $\mathcal{G}$ is the foliation by oriented geodesics on $T^1 \mathbb{H}^2$, where $T^1\mathbb{H}^2 $ is trivialized as a product $\mathbb{H}^2 \times S^1$ with the weak-stable leaves horizontal.   Points in the complement of the diagonal (the graph of $f$ in the boundary $S^1 \times S^1$ of $T^1\mathbb{H}^2$, shown in orange on the figure), give the Hopf coordinates on the space of oriented geodesics.  See Figure \ref{fig_geodesic_flow}, left.
\end{rem}

\begin{figure}[h]
\centering 
\begin{overpic}[width=12cm]{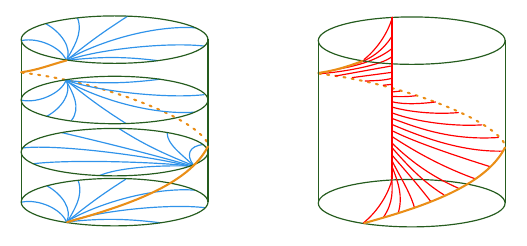}

\put(-8,7){{\small $\wt M \times \{0\}$}}

\put(-8,17){{\small $\wt M \times \{p\}$}}

\put(-8,38){{\small $\wt M \times \{1\}$}}

\put(10,2){{\tiny $f(0)$}}

\put(36,13){{\tiny $f(p)$}}

\put(74,45){{\small $\xi$}}

\end{overpic}
\caption{{\small The special case $f= id$, $\dim(M) = 2$. On the left, some leaves $\wt M \times \{p\}$ of $\cF^h$ with  geodesics pointing towards $f(p)$, the graph of $f : S^1 \to \partial \wt M$ is the orange line. On the right, a leaf of $\cF^v$ consisting of geodesics from $\xi$ to $f(p)$ at each $\wt{M} \times \{p\}$.}}
\label{fig_geodesic_flow}
\end{figure}
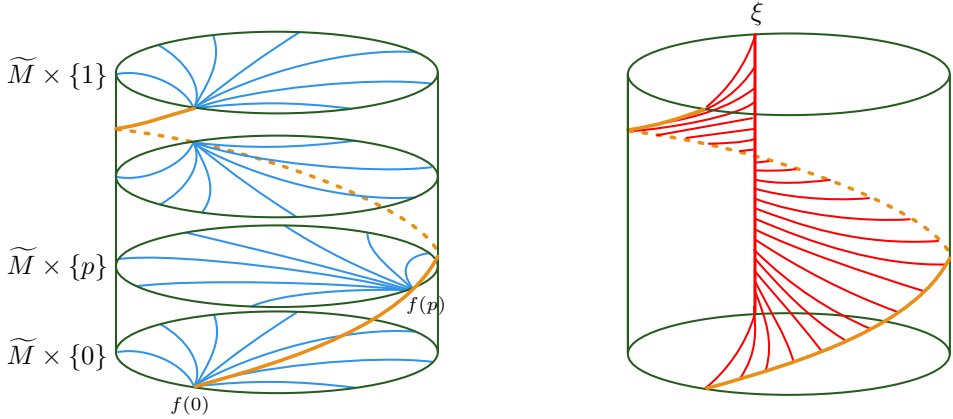

We now define a 2-dimensional foliation $\cF^v$ in $\wt M \times S^1$ transverse to $\cF^h$.
For each $\xi \in \partial \wt{M}$, let $L_\xi$ denote the union (over all $p \in S^1$) of leaves of $\mathcal{G}$ with negative endpoint $\xi$ in $\wt{M} \times \{p\}$. Note that there is exactly one such leaf of $\mathcal{G}$ in each $\wt{M} \times \{p\}$ when $f(p) \neq \xi$, and if $f(p)=\xi$ then there is no such leaf.  
It follows from this uniqueness, together with the continuity of $f$, 
that each set $L_\xi$ is locally a plane, containing one continuously varying geodesic in each slice $\wt{M} \times \{p\}$.  
Formally, in a neighborhood of a leaf of $\mathcal{G}$ one has charts trivializing the foliations $\cF^v$ and $\cF^h$ by varying the endpoints of each.
Thus, these $L_\xi$ define leaves of a 2-dimensional foliation $\cF^v$ transverse to $\cF^h$, whose leaves are saturated by leaves of the 1-dimensional foliation $\mathcal{G}$. In the special case described by Remark \ref{rem_surface} and shown in Figure \ref{fig_geodesic_flow} right, each $L_\xi$ is connected, however in general it may not be and instead forms a union of leaves of $\cF^h$.  

We claim that $\mathcal{G}$ descends to a foliation of $N_\rho$. This follows easily from equivariance: given a leaf $l$ of $\mathcal{G}$ in  $\wt{M} \times \{p\}$ with negative endpoint $\xi$, and $\gamma \in \pi_1(M)$, we have that $\gamma(\wt{M} \times \{p\})$ is the leaf $\wt{M} \times \{\rho(\gamma)(p)\}$ and $l$ is mapped to the geodesic with negative endpoint $\gamma(\xi)$ and positive endpoint $\gamma(f(p)) = f(\rho(\gamma)(p))$, which is a leaf of $\mathcal{G}$.  By construction, we also have that $\cF^h$ and $\cF^v$ descend to $N_\rho$.  Abusing notation slightly, we refer to these induced  foliations on $N_\rho$ also by $\cF^h, \cF^v$ and $\mathcal{G}$. 
(The reader following the example of Remark \ref{rem_surface} will note that
in that case $N_\rho = T^1M$ and $\cF^h$ and $\cF^v$ are the stable and unstable foliations of geodesic flow.)

Let $\wt \varphi^t$ be a parameterization of the foliation $\mathcal{G}$ in $\wt M \times S^1$ so that on each leaf of $\cF^h$ it is unit speed towards $f(p)$ in the negatively curved metric on $\wt M$ lifted from $M$.  Let $\varphi^t$ denote the induced flow on $M$. 

\begin{proposition} \label{prop-susp}
The flow $\varphi^t$ is not a suspension. 
\end{proposition}
\begin{proof}
This follows by a simple computation of the fundamental group.   Codimension one Anosov homeomorphisms are conjugate to toral automorphisms (\cite{Franks, Newhouse}, see also \cite{Sumi}), and so the fundamental group of a suspension is solvable. This is not the case for $\pi_1(N_\rho)$, which contains nonabelian free groups.   Alternatively, in dimension $\geq 4$, a manifold with a suspension flow contains a $\mathbb{Z}^3$, and $\pi_1(N_\rho)$ does not. 
\end{proof}

In the next section we show that $\varphi^t$ is topologically Anosov and use this to prove Corollary \ref{cor-CTmaps}.  
In Section \ref{section-smooth}, we show that for a suitable choice of smooth structure, $\varphi^t$ can be perturbed to a flow that is Anosov. This is independent of Section \ref{sec_topological}, so these sections can be read in any order.

\section{Proof of Theorem \ref{thm_main} Topological case, and Corollary \ref{cor-CTmaps}} \label{sec_topological}

 \begin{definition}\label{def_TAF} (Topological Anosov flow)
 Let $N$ be a compact manifold with a Riemannian metric.  A flow $\phi^t$ on $N$ with no stationary points is  
 a  \emph{topological Anosov flow} if it satisfies the following: 
 \begin{enumerate}[label = (\roman*)]
\item (Transverse foliations) \label{item_TAF_weak_foliations}  There exist topologically transverse $C^0$ foliations $\cF^{s}$ and $\cF^{u}$ whose leaves are saturated by orbits of $\phi^t$ and intersect along orbits of $\phi^t$.
\item (Asymptotic orbits along leaves) \label{item_TAF_forward_asymptotic}  For any $\eps>0$ sufficiently small, there exists $\delta > 0$ such that if 
$y \in \cF_\delta^{s}(x)$ (resp.~$y \in \cF_\delta^{u}(x)$), there is a standard reparameterization $\sigma$ such that $d\left(\phi^t(x), \phi^{\sigma(t)}(y)\right) < \eps$ for all $t>0$ and tends to 0 as $t\to +\infty$ (resp.~$t\to -\infty$). 
\item (Separation along leaves) \label{item_TAF_backwards_expansivity}
There
exists $\eps>0$ so that for any $\delta>0$ sufficiently small, if $y \in \cF^s_\eps(x)$ (resp.~$y \in\cF^u_\eps(x)$) is not in the same $\eps$-local orbit as $x$, then for any standard reparameterization $\sigma$, there exists $t<0$ (resp.~$t>0$) with $d(\phi^t(x), \phi^{\sigma(t)}(y)) > \delta$. 
\end{enumerate}
 \end{definition}
 
\noindent A {\em standard reparameterization} is a nondecreasing function $\sigma: \bR \to \bR$ with $\sigma(0) = 0$. For a foliation $\cF$, the notation $\cF_\delta(x)$ means the {\em $\delta$-local leaf} of $\cF$, which is the connected component containing $x$ of the intersection of the leaf $\cF(x)$ containing $x$, with the $\delta$-ball about $x$. 
While the definition is framed in terms of distances, on a compact manifold the property of being topologically Anosov is independent of the choice of metric.   Note also that the separation property \ref{item_TAF_backwards_expansivity} implies that topologically Anosov flows are expansive.  See \cite[Ch. 1]{BM} for more discussion.
 
We will show that the flow $\varphi^t$ from Section \ref{sec_construction} is topologically Anosov with $\cF^h$ playing the role of $\cF^s$  and $\cF^v$ as $\cF^u$.  By construction these foliations are topologically transverse and saturated by orbits.  Items \ref{item_TAF_forward_asymptotic} and \ref{item_TAF_backwards_expansivity} are already satisfied for $\cF^h$ because orbits along $\cF^h$-leaves are geodesics with a common forward endpoint, which are forwards-asymptotic in any manifold of strict negative curvature.

So it remains to verify \ref{item_TAF_forward_asymptotic} and \ref{item_TAF_backwards_expansivity} for $\cF^v$.  
To do this, we will first find a periodic orbit $\alpha$ (in fact we find many, see Lemma \ref{lema-periodicorbits}), then show \ref{item_TAF_forward_asymptotic} and \ref{item_TAF_backwards_expansivity} for $\cF^v$ locally along $\alpha$ by studying the transverse holonomy of $\cF^h$, and then use minimality of the action $\rho$ to propagate the behavior everywhere. This argument is largely inspired by one due to Calegari for flows on 3-manifolds (see \cite[ch. 5]{Calegari}).  Related arguments, in a different context (that of codimension 1 dominated splittings), appear in the proof of Corollary 1.5 of \cite{PujSam}.

\begin{lemma}\label{lema-periodicorbits}
For every $\gamma \in \pi_1(M)$ there is some\footnote{In fact, with more care, one could show this $k$ is uniform, but this will also follow from Corollary \ref{cor-CTmaps}.}  $k$ so that $\rho(\gamma^k)$ has fixed points in $S^1$, each of which corresponds to some periodic orbit of $\varphi^t$.
\end{lemma} 
\begin{proof}
The image of $f$ is compact and $\pi_1(M)$-invariant (by equivariance).  Since $\pi_1(M)$ acts minimally on $\partial \wt{M}$, it follows that $f$ must be surjective. 

Let $\gamma$ be a nontrivial element of $\pi_1(M)$, with attracting and repelling fixed
points $\gamma^+, \gamma^-$ (respectively) in $\partial \wt{M}$. Then  
 $B := f^{-1}(\{ \gamma^+, \gamma^- \})$ is a nonempty, closed set.  
Let $J$ be a connected component of $S^1 \setminus B$ with endpoints $a, b$. 
We claim $J$ is periodic.  First, if $f(a) = f(b)$, assume without loss of generality that they are equal to $\gamma^+$.
Then $\gamma^n(f(J))$ has points converging to $\gamma^-$, as $n \to -\infty$, so its
diameter is bounded below by a positive number for all $n \leq 0$.  
Since $f$ is uniformly continuous, $\rho(\gamma)^n(J)$ cannot all be disjoint, so $J$ is periodic under $\rho(\gamma)$.
(In fact, this now gives a contradiction since $\gamma^-$ is not in the closure of $f(J)$). 
When $f(a) \neq f(b)$, then $\gamma^n(f(J))$ limits on both $\gamma^-$ and $\gamma^+$, and therefore has diameter bounded below, again implying that $J$ is periodic under $\rho(\gamma)$.

Thus, some power $\rho(\gamma^k)$ has a fixed point $p \in S^1$.  
This means that $\gamma^k$ acts on $\wt{M} \times S^1$ preserving the leaf $\wt M \times \{p\}$ and acting on this leaf by the corresponding hyperbolic isometry.  Since $\gamma^k(f(p)) = f (\rho(\gamma^k)(p)) = f(p)$, the geodesic axis $\alpha$ of $\gamma^k$ in $\wt M$ is a 
leaf $\alpha \times \{ p \}$ of $\mathcal{G}$ in $\wt M \times \{p\}$, which is invariant under $\gamma^k$, and its quotient in $N_\rho$ is a periodic orbit of $\varphi^t$.  
\end{proof}

\subsection{Transverse holonomy of $\cF^h$} 
Let $\alpha$ be a periodic orbit of $\varphi^t$ (which exists by Lemma \ref{lema-periodicorbits}), and $B_\alpha$ an embedded tubular neighborhood of $\alpha$ in $N_\rho$. 
Let $T$ be a transversal to $\varphi^t$ based on $\cF^h(\alpha)$, chosen so that $\cF^h$ and $\cF^v$ trivially foliate $T$. Fix coordinates $T \cong D \times [0,1)$, where $D = (-1,1)^{n-1}$ denotes the open disk, so that $\alpha \cap T$ is at $(0, 0)$ and the intersection of $T$ with $\cF^h$ is the horizontal foliation and $\cF^v$ is the vertical.    (See Figure \ref{fig_sawblade} for an image of the lift to $\wt M \times S^1$.)

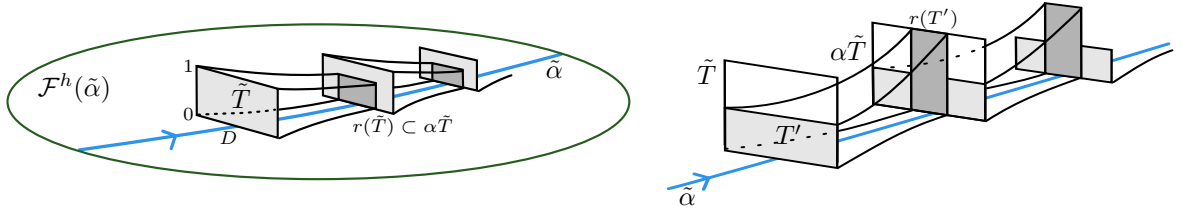
\begin{figure}[h]
\centering 
\begin{overpic}[width=16cm]{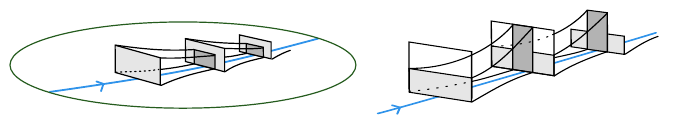}
\put(4,10){{\small $\cF^h(\tilde \alpha)$}}
\put(46,11.5){{\small $\tilde \alpha$}}
\put(20,9){{\small $\tilde T$}}
\put(19,6){{\tiny $D$}}
\put(30,7){{\tiny $r(\tilde T) \subset \alpha \tilde T$}}
\put(16,12){{\tiny $1$}}
\put(16,8){{\tiny $0$}}

\put(58.5,11){{\small $\tilde T$}}
\put(65,6){{\small $ T'$}}
\put(57,1){{\small $\tilde \alpha$}}
\put(69.5,13){{\small $\alpha \tilde T$}}
\put(76,16){{\tiny $r(T')$}}

\end{overpic}
\caption{{\small Transversals and first-returns along $\wt \alpha$. In the left one sees the configuration leading to a contradiction and in the right the actual configuration.}}
\label{fig_sawblade}
\end{figure}

Let $T' \subset T$ be the subrectangle containing $\alpha \cap T$ where the first-return to $T$ along $\phi^t$-orbits is defined and continuous.  We choose $T$ small enough so that the union of orbit segments from $T'$ to $T$ is contained in the tubular neighborhood $B_\alpha$.   Let $r: T' \to T$ denote this first return. 

Since orbits forward converge along leaves of $\cF^h$, the set $T'$ is of the form
$D \times [0,a)$ for some $a \leq 1$, and $r(T') = D' \times [0,b)$ where $D'$ has closure strictly contained in $D$ and $b \in (0,1]$.

\begin{lemma}[Transverse contraction in the past along $\alpha$]  \label{lem_periodic}
In the coordinates above, $a <1$ and $b=1$.  
In other words, the inverse of $r$ is defined and continuous on $D' \times [0,1)$ and strictly contracts the $[0,1)$ factor. Moreover,  for any $c \in T$ on the $\cF^v$ leaf of $\alpha$ 
we have $r^{-n}(c) \to 0$ as $n \to \infty$. 
\end{lemma} 

\begin{proof} 
Suppose for contradiction this is not the case. 
Then $r$ is defined and continuous on all of $T$ and maps $T$ to $D' \times [0,b)$ for some $b \leq 1$. (See Figure \ref{fig_sawblade} for a picture lifted to $\wt M \times S^1$.)  Let $S \subset B_\alpha \subset N_\rho$ denote the union of orbit segments from $T$ to their first return in $r(T)$.  
We lift this picture to $\wt{M} \times S^1$: 
Let $\wt \alpha$ be a lift of $\alpha$ to $\wt M \times S^1$,  let $\wt{M} \times \{p\}$ be its $\cF^h$ leaf, and $\wt T$ a lift of $T$ meeting $\wt \alpha$, so $T$ intersects the leaves $\wt M \times \{s\}$ for $s$ in some interval $[p,q) \subset S^1$, and let $\wt S$ denote the connected component of the pre-image of $S$ containing $\wt \alpha$.  
Here we abuse notation and also use $\alpha$ for the element
of $\pi_1(M)$ associated to $\alpha$ and leaving $\wt \alpha$ invariant.
Since $S$ is contained in the tubular neighborhood $B_\alpha$, for all $\gamma \in \pi_1(M)$ not equal to a power of $\alpha$, the image of $\wt S$ under the action of $\gamma$ on $\wt M \times S^1$ is disjoint from $\wt S$.  

However, since $\rho$ is minimal, there exists $\gamma \in \pi_1(M)$, not a power of $\alpha$, such that $\rho(\gamma)(p) \in (p,q)$.  Since orbits converge forward along $\cF^h$ leaves, the image of $\wt \alpha$ under $\gamma$ intersects $\wt S$.  But this means that $\gamma \wt S \cap \wt S \neq \emptyset$, contradiction. 

Since the statement applies to any sufficiently small transversal, we obtain the final assertion of the lemma.  
\end{proof}

\begin{lemma}  \label{lem_expansive} 
Property \ref{item_TAF_backwards_expansivity} holds for $\cF^v$. 
\end{lemma} 

\begin{proof} 
Let $T_0$ and $T_1$ be transversals based on {\em either side} of $\cF^h(\alpha)$ with the property from Lemma \ref{lem_periodic}. Let $T$ denote their union, and let $T' \subset T$ be the subrectangle where the first return map is defined and continuous.  Adapting the notation from above, let $S$ denote the union of orbit segments from $T'$ to their first return image (so $S$ is now an embedded  {\em neighborhood} of $\alpha$ in $N_\rho$), and let $\wt S$ be the lift of $S$ to $\wt{M} \times S^1$ containing $\wt \alpha$.  Take any $\eps$ small enough so that $S$ contains an $\eps$-neighborhood of $\alpha$.   We will show separation with $\delta = \eps$. 

Suppose $x \in N_\rho$, take $y \in \cF^v_\eps(x)$ and not in the same local orbit as $x$, and lift these to $\wt x \in \cF^v(\wt y)$ in $\wt{M} \times S^1$. Recall that $\wt \alpha$ is contained in $\wt M \times \{ p \}$.  
Let also $\wt T'$ be a lift of $T'$ intersecting $\wt \alpha$. By minimality of $\rho$, there exists $\gamma \in \pi_1(M)$ such that $\rho(\gamma)(p)$ lies in the small interval between $\cF^h(\wt x)$ and $\cF^h(\wt y)$, so, up to replacing our original lifts with their translates by the action of $\gamma^{-1}$ on $\wt{M} \times S^1$, we can assume that $\cF^h(\wt \alpha) = \wt{M} \times \{p\}$ intersects a transversal $\tau$ (of length $<\eps$) to $\cF^h$ between $\wt x$ and $\wt{y}$; let $z$ be this point of intersection.  Thus $\wt \varphi^t(z)$ is forward asymptotic to $\wt \alpha$, and
so intersects $\alpha^k(\wt T')$ for all $k$ sufficiently large.  

If the orbits of $\wt x$ and $\wt{y}$ do not $\eps$-separate in the future, then they 
eventually both meet an $\alpha^k$  translate of $\wt T'$  inside $\wt S$.   
This is because the forward orbit of $z$ always keeps intersecting $\alpha^k(\wt T')$, and this orbit is between the orbits of $\wt x$ and $\wt y$. Since Lemma \ref{lem_periodic} applies to {\em every} small transversal to $\alpha$, then in this latter case, by applying further iterates of $r$, we see that the future orbits of $\wt x$ and $\wt{y}$ eventually leave a translate of $\wt T'$ on opposite sides of $\cF^h(\wt \alpha)$, and thus in fact do $\eps$-separate. 
\end{proof} 

\begin{rem} \label{rem_separate_forever}
In fact, the structure of $S$ shows that after the orbits of $\wt x$ and $\wt{y}$ separate, they can never return within distance $\eps$ from each other (in $\wt M$).  Formally speaking, 
there is some time $t_0$ such that for all $t>t_0$, for any standard reparameterization $\sigma$, we have $d(\wt{\varphi}(\wt x),\wt{\varphi}^{\sigma(t)}(\wt{y})) > \eps$.   
\end{rem} 

To prove Property \ref{item_TAF_forward_asymptotic}, we need a preliminary lemma: 
\begin{lemma}[Separation in uniform time]  \label{lem_uniform_time}
Let $\eps$ be as in the proof of Lemma \ref{lem_expansive}, and $0 < k < \eps/2$. There exists $t(k)$ such that for any points 
$y \in \cF^{v}_{\eps/2}(x)$ with $d(y,x) \geq  k$ and not in the same local orbit, for any 
 standard reparameterization $\sigma$, there is $0 < t<t(k)$ with $d(\varphi^t(x), \varphi^{\sigma(t)}(y)) > \eps/2$. 
\end{lemma}

\begin{proof} 
This is a standard consequence of compactness (the argument
is made in $N_\rho$).  Suppose for contradiction we had a sequence $t_n \to \infty$ and $y_n \in \cF^{v}_{\eps/2}(x_n)$,  with $d(y,x) \geq  k$
whose orbits fellow-travel at distance $\leq \eps/2$ up to time $t_n$.  Consider an accumulation point $(x, y)$ of the sequence $(x_n, y_n)$ 
 in $N_\rho \times N_\rho$.  
We have that $y \in \overline{\cF^{v}_{\eps/2}(x)}$ and $d(x,y) \geq k$. Hence by Lemma \ref{lem_expansive}, the orbits of $x, y$ will $\eps$-separate in some finite time $t_0$, contradicting that the orbits of arbitrarily close points $x_n, y_n$  will $\eps/2$-fellow-travel for time up to $t_n>t_0$.  
\end{proof} 

\begin{lemma}[Asymptotic orbits]
Property \ref{item_TAF_forward_asymptotic} holds for $\cF^v$.  
\end{lemma} 

\begin{proof}
Let $x \in N_\rho$.  We first claim that, for any $\eps$ small enough to apply Lemma \ref{lem_uniform_time}, and $k< \eps/2$ there exists $\delta > 0$ so that 
if $y \in \cF_\delta^{v}(x)$, then there is a standard reparameterization $\sigma_0$ such that $d(\varphi^t(x), \varphi^{\sigma_0(t)}(y))< k$ for all $t<0$.  To see this, suppose for contradiction it is not the case, and take $y_n \to x$ in $\cF^v(x)$ so that no such reparameterization exists.  For fixed $y_n$ 
let  $t_n$ be the first negative time where the backwards orbits of $x$ and $y_n$ fail to $k$-fellow-travel, that is 

\[ t_n = \inf_{\sigma} \sup \{ t < 0 :  d(\varphi^s(x), \varphi^{\sigma(s)}(y_n)) < k,\  \forall s > t, s < 0\}. \]
Here the infimum is taken over all standard reparameterizations $\sigma$.  

By continuity of $\varphi$, there is a point $y'_n$ on the negative orbit of $y$ with $d(\varphi^{t_n}(x), y'_n) = k$,  and we have $t_n \to -\infty$.  But this contradicts the fact that any two points distance $k$ apart will $\eps/2$-separate in uniformly bounded time by Lemma \ref{lem_uniform_time}.  This proves the claim. 

Now we show orbits converge in the past. Suppose that $y \in \cF_\delta^{v}(x)$.  If the orbits of $x, y$ do not converge in the past, then 
then we can find some $k_0 > 0$ and $t_n \to -\infty$ such that  $d(\varphi^{t_n}(x), \varphi^{\sigma(t_n)}(y)) \geq k_0$ for any standard reparameterization, in particular for the
reparameterization $\sigma_0$ from the claim above, which has the property $d(\varphi^t(x), \varphi^{\sigma_0(t)}(y))< k$ for all $t<0$.  This implies that $k_0 < k$.  

Let $z_n = \varphi^{t_n}(x)$ and $w_n = \varphi^{\sigma_0(t_n)}(y)$. 
Rephrasing the statements above, we have $d(z_n,w_n) = k_0$, and moreover $z_n$ and $w_n$ are
in the same local $\cF^v$ leaf. In addition, 
\[d(\varphi^s(z_n), \varphi^{\sigma_0(s+t_n)-\sigma_0(t_n)}(w_n)) \ < 
 k  \ < \ \eps/2\]
for all $0 \leq s \leq |t_n|$,
and  $s \mapsto \sigma_0(s+t_n)-\sigma_0(t_n)$ is a standard reparameterization.

Since $|t_n| \to +\infty$, and $d(z_n,w_n) = k_0$ for all $n$, this
again contradicts Lemma \ref{lem_uniform_time}.  
\end{proof} 

This concludes the proof that $\varphi^t$ is topologically Anosov.    \qedhere

\begin{rem}[Free homotopy classes of periodic orbits] 
Our argument gives a negative answer to
 \cite[Question 1.9]{BBGH}.  This question asked whether the existence of a pair of periodic orbits $\alpha, \beta$ of an Anosov flow, with $\alpha$ freely homotopic to $\beta^{-1}$, implies that the weak stable and weak unstable foliations have the same dimension.  The proof of Lemma $\ref{lema-periodicorbits}$ (and minimality of $\rho$) 
 shows that fixed points of $\rho(\gamma)$ on $S^1$ alternate between preimages of $\gamma^+$ and $\gamma^-$ under $f$, so are alternating attractors and repellers.  
Thus, if $p_1, p_2$ are consecutive fixed points of $\rho(\gamma)$, the associated periodic orbits of $\varphi$ are 
freely homotopic to the inverses of each other.
This also gives an alternative proof that $\varphi$ is not a suspension, since suspension flows do not have pairs of orbits which are freely
homotopic to the inverse of each other.
\end{rem}


\subsection{Cannon-Thurston maps are finite to one}
We conclude this section with the proof of Corollary  \ref{cor-CTmaps}.  

\begin{proof}[Proof of Corollary \ref{cor-CTmaps}]
Let $\varphi^t$ be the flow constructed in Section \ref{sec_construction}.  Since it is topologically Anosov, it is expansive; meaning that, there exists $\eps>0$ such that any two points whose orbits can be reparameterized to $\eps$-fellow travel for all time, are in fact on the same orbit.   

By Remark \ref{rem_separate_forever}, orbits which are not on the same weak stable leaf for $\wt \varphi^t$ on $\wt M \times S^1$ in fact (up to descreasing $\eps$) will $\eps$ separate at some point 
and then stay $\eps$-separated for the future.  
Suppose for contradiction that some point $\xi \in \partial \wt M$ has more than $C/\eps$ preimages by $f$, where $C$ is the maximum length of a fiber $S^1$.  Let  $p_1, p_2,...p_n$ denote these preimages.

Fix some $\beta \neq f(p_i)$ and consider orbits $\alpha_i := \wt \varphi^t(x_i) \subset \wt M \times \{p_i\}$ with backwards endpoint $\beta$ in each leaf $\wt M \times \{p_i\}$.  These meet every fiber at $n$ points, so (by piegonhole) there will be some $i, j$ and a sequence of times, going to infinity, along these orbits where some fixed $\wt \varphi^t(x_i)$ and $\wt \varphi^s(x_{j})$ are in the same $S^1$ fiber and distance $< \eps$ apart, contradiction.  
\end{proof}

\section{Proof of Theorem \ref{thm_main}: Smooth statement} \label{section-smooth}

In this section we assume that $M$ has a hyperbolic structure and that the image of the non-injective points of $f$  in $\partial \wt{M}$  has zero Lebesgue measure.  Our proof could likely be adapted to the negative curvature case (using the quasiconformal rather than conformal class of visual metrics on the boundary), but we do not pursue this here.   

\subsection{Smooth structure on $N_\rho$ and metric computation.}  \label{sec_smooth_structure} 
Our first goal is to equip $N_\rho$ with a smooth structure where $\cF^h$ is $C^1$ and  show that $\varphi^t$ has uniform expansion transverse to $\cF^h$. This will be done using an idea of Tholozan  \cite{Tholozan}. 


A key ingredient is the following basic observation about transformation of visual measures in hyperbolic $n$-space, which forms the basis of Patterson--Sullivan theory (see \cite[Section 2]{Su} or \cite[Proposition 3.39]{Qu}): if $\mu_x$ denotes the visual measure on $\partial{\wt M}$ from the basepoint $x \in \wt M$, then, given $x,y \in \wt M$ and $\xi \in \partial \wt M$:
\[ \frac{d \mu_x }{d \mu_y } (\xi) = e^{-\delta{b_\xi(x,y)}} \] 
where $\delta = \dim(M) - 1 = \dim \partial(\wt M)$ and $b_\xi(\cdot, \cdot)$ is the Busemann function at $\xi$, measuring the signed distance between the horospheres based at $\xi$ containing $x$ and the horosphere containing $y$.  By convention the sign of $b_\xi(x,y)$ is positive when $y$ is closer to $\xi$ than $x$.

Fix a basepoint $o \in \wt M$. 
For an interval $I \subset S^1$, define 
\[ \nu_o(I) = \mu_o(f(I)).\] 

If $I$ is closed, then $f(I)$ is compact, so measurable. If $I$ is open, and $J$ its complement, then $f(I)$ differs from the complement of $f(J)$ which is open, by a subset of the image of the non-injectivity points, which are a null set. Thus, $f(I)$ is also (Lebesgue) measurable\footnote{In fact, if we use Corollary \ref{cor-CTmaps} which says $f$ is finite-to-one, the Lusin---Novikov theorem immediately implies that $f$ maps Borel sets to Borel sets.}. This implies that $\mu_o(f(I))$ is well defined.  Since the images of non-injective points of $f$ form a $\mu_o$-null set ($\mu_o$ is in the Lebesgue measure class), $\nu_o$ is countably additive so defines a Borel measure on $S^1$. 
One should think of $\nu_o$ as the ``pullback" of $\mu_o$ under $f$.

Since $\rho$ is minimal, the orbit of any point is infinite and dense and thus $\nu_o$ is nonatomic and  has full support. Thus, we can define a system of local coordinate charts $\zeta_{ab}$ in $S^1$ from open sets $(a,b) \subset S^1$ to 
 $\mathbb{\R}$ by $\zeta_{ab}(s) := \nu_o([a,s])$.  Also, the change-of-coordinate overlap maps are (in each connected component of the overlapping) translations of intervals in $\mathbb{R}$, so these charts give a smooth structure.  
 
 \begin{lemma}[Tholozan's trick]
 With respect to this smooth structure, the action of $\pi_1(M)$ on $S^1$ 
given by $\rho$ is by $C^1$ diffeomorphisms\footnote{In fact, Tholozan shows this can be made by $C^{1+1/2}$ diffeomorphisms \cite{Tholozan-inprep}.}. 
 \end{lemma}

\begin{proof}

For any measurable $A \subset S^1$ and $\gamma \in \pi_1(M)$, we have 
\[\rho(\gamma)_\ast \nu_o(A) = \nu_o(\rho(\gamma)^{-1}(A)) = \mu_o(\gamma^{-1} f(A)) =\mu_{\gamma o}(f(A)) = \nu_{\gamma o}(A).\] 

Let $h(\xi) = \frac{d\mu_{\gamma_o}}{d\mu_o}(\xi)$.
We claim that 
\[\frac{d\nu_{\gamma o}}{d\nu_o}(q) = h(f(q))\]
To prove that we compute:
\[\nu_{\gamma o}(A) \ = \ \mu_{\gamma o}(f(A)) \ = \ 
\int_{f(A)} h(\xi) d\mu_o(\xi) \ = \ \int_A h(f(q)) d\nu_o(q).\]
Therefore 
\[\frac{d\nu_{\gamma o}}{d\nu_o}(q) \ = \ h(f(q)) \ = \ 
\frac{d\mu_{\gamma o}}{d\mu_o}(f(q))\]
as desired. As a consequence
\begin{equation}  \label{eq_change-of-coord}
\frac{d \rho(\gamma)_\ast \nu_{o}} {d \nu_o } (p)=  \frac{d \mu_{\gamma o}}{d \mu_o }( f(p)) = e^{-\delta{b_{f(p)}(\gamma(o), o)}}
\end{equation} 

Since the family of Busemann functions $b_{\xi}$ is continuous in $\xi$, and $f$ is continuous, the Radon Nikodyn derivative $\frac{d \rho(\gamma)_\ast \nu_{o}} {d \nu_o } $ is continuous.  But this is just the derivative of the transformation $\gamma$ with respect to our smooth structure.  
\end{proof}

We now equip $\wt{M} \times S^1$ with the product smooth structure given by the standard structure on $\wt{M}$, and the structure just constructed on the $S^1$ factor.  This descends to give a $C^1$ structure on $N_\rho$, for which the horizontal foliation is transversely $C^1$.   In particular, for a path $c: [0,t] \to \wt{M} \times \{p\}$, the transverse holonomy 
along $c$ defines a $C^1$ function
from $c(0) \times S^1$ to $c(t) \times S^1$.

\begin{notation}
For $x \in \wt{M} \times S^1$, we denote by $\hol^{\varphi}_{x,t}$ the transverse holonomy of the foliation
$\cF^h$ along the path $s \mapsto \varphi^s(x)$, for $s \in [0, t]$,
computed from $\varphi^0(x) \times S^1$ to $\varphi^t(x) \times S^1$.
\end{notation} 

Our next goal is to show that this holonomy is uniformly expansive.  To do this, 
define a ``piecewise" Riemannian metric on $\wt{M} \times S^1$ as follows:  Fix a fundamental domain $D$ for the action of $\pi_1(M)$ on $\wt M$ with compact closure, containing the basepoint $x$ used to define the smooth structure.   We declare the factors $\wt{M}$ and $S^1$ to be orthogonal everywhere, and put the hyperbolic metric along each horizontal slice $\wt{M} \times \{p\}$.   Finally, along each $S^1$ fiber over $\gamma D$, for $\gamma \in \pi_1(M)$, we measure length in the coordinate given by $\rho(\gamma)_\ast \nu_o$, so the length of the interval $[a,b] \subset S^1$ in a fiber above a point of $\gamma D$ is  $\rho(\gamma)_\ast \nu_o[a, b] = \nu_o(\rho(\gamma^{-1})([a,b]))$.   

Although this ``metric" is not continuous, it defines a (discontinuous) length $|| \cdot ||$ for vectors in $T(\wt{M} \times S^1)$, which is
 invariant under the diagonal action of $\pi_1(M)$ on $\wt M \times S^1$, so descends to a discontinuous  
 metric on $N_\rho$.  Moreover, this metric $|| \cdot ||$ has the property that for any Riemannian metric $g$ on $N_\rho$, there exists $k= k(g) >1$ such that 
\[  1/k || v ||_g \leq ||v || \leq k || v ||_g \] 
for all $v \in T N_\rho$.  This can be seen by lifting $g$ to a metric $\wt g$ on $\wt{M}$:  since $D$ has compact closure and the discontinuous metric in the $S^1$ fibers is constant in $D$, it follows that there is a constant $k$ such that $1/k || v ||_{\wt g} \leq ||v || \leq k || v ||_{\wt g}$ holds for all $v$ in the tangent bundle over $D$, and this estimate descends to $N_\rho$.

Using this, we verify the following: 
\begin{lemma}[Transverse contraction and expansion] \label{lem_transverse_expansion}
Let $X$ be the generating vector field for $\varphi^t$, and let $E^h \subset TN_\rho$ denote the tangent bundle of the horizontal foliation $\cF^h$. This splits as a direct sum $E^h = E^s \oplus X$ which is invariant under the flow.

Furthermore, for any Riemannian metric $g$ on $N_\rho$ there exists $k_1, k_2, k_1', k_2' > 0$ such that: 
\begin{equation} \label{eq_contraction}
\begin{split}
 & || \varphi^t_\ast(v) ||_g < k'_1 e^{-tk'_2} || v ||_g \text{ for all } v \in E^s, t>0, \text{ and } \\
 & || \mathrm{D} \hol^{\varphi}_{y,t}(v)||_g  > k_1 e^{tk_2} ||v||_g \text{ for all } y \in N_\rho, t>0.  
  \end{split}
\end{equation}    
\noindent
where in the second formula $v$ is tangent to the circle foliation in $N_\rho$.

\end{lemma}
Note that, since $\varphi$ is smooth along $\cF^h$ leaves, $\varphi^t_\ast(v)$ makes sense. 

\begin{proof}  
The fact that $E^h$ splits as $E^s \oplus X$ is because $X$ restricts to a (hyperbolic) flow along geodesics with common forward endpoint on each leaf.  
The estimate for $E^s$ now follows from the fact that, in the leafwise hyperbolic metric, lengths are contracted exponentially; and any other metric differs by a constant.  

To estimate the derivative of the holonomy, we first do a computation in the piecewise Riemannian metric $|| \cdot ||$ on $\wt M \times S^1$ defined above.   Fix a horizontal leaf $\wt M \times p$ and suppose that  $\alpha, \beta \in \pi_1(M)$ are the elements such that $y \in \alpha(D) \times\{p\}$ and $\wt \varphi^t(y) \in \beta(D) \times \{p\}$, for some $t<0$.  
Let $v$ be a tangent vector to the $S^1$ fiber at $y$.  The definition of the metric and 
the Radon Nikodym derivative gives
\[ || \mathrm{D} \hol^{\varphi}_{y,t} (v)|| = \frac{d \beta_\ast \nu_{o}} {d \alpha_\ast \nu_o }(p) ||v|| = e^{-\delta b_{f(p)}(\beta(o), \alpha(o))} ||v||  \]
where the last equality follows from Equation \eqref{eq_change-of-coord}. 

The distance in $\wt M$ between $\alpha(o)$ and $y$ is bounded above by $\diam(D)$, as is the distance between $\beta(o)$ and $\wt\varphi^t(y)$.  
Since $y$ and $\wt\varphi^t(y)$ lie along a geodesic with endpoint $f(p)$, we also have that the horosphere based at $f(p)$ and containing $\alpha(o)$ is distance at most $\diam(D)$ from $y$; and have the same distance bound between the horospheres of $\beta(o)$ and $\varphi^t(y)$.  Thus, 
\[ b_{f(p)}(\beta(o), \alpha(o)) \geq b_{f(p)}(\wt\varphi^t(y),y)- 2\diam(D).\]  

Since $\varphi^t$ is unit speed in the hyperbolic metric, and on the leaf containing $y$ it is along geodesics with positive endpoint $f(p)$, we have  $b_{f(p)}(\wt\varphi^t(y),y) = -t$.  Recall that $t$ was chosen negative, so $b_{f(p)}(\wt \varphi^t(y), y)$ is positive. 
Therefore
\[-b_{f(p)}(\beta(o),\alpha(o)) \ \leq \ -b_{f(p)}(\wt \varphi^t(y),y) + 2\diam(D) \ = \
t + 2\diam(D),\]
which implies that for negative times we have \
\[||\mathrm{D} \hol^{\varphi}_{y,t}(v)|| \ \leq \ e^{2\delta \diam(D)} e^{\delta t}||v||.\]
Using that for every $y,t$ we have $\hol^{\varphi}_{\varphi^t(y), -t} \circ \hol^{\varphi}_{y,t}= \mathrm{id}$, and modifying the constants above so strict inequality holds, we obtain the inequality in \eqref{eq_contraction} for $|| \cdot ||$.  

Since, for any $v \in TM$ we have a constant $k>0$ such that $1/k || v || \leq  || v ||_g \leq k || v ||$, the desired inequality for a general metric $|| \cdot ||_g$ follows.  
\end{proof}

\subsection{Smooth Anosov perturbation} 
As before, we let $X$ denote the vector field generating $\varphi^t$ on $N_\rho$.  While $X$ is smooth along leaves of $\cF^h$, it varies only continuously in the transverse direction.   Following a strategy of Asaoka (see \cite[pages 457-458]{Asaoka}), we show the following. 

\begin{lemma} 
There exists a sequence of $C^1$ vector fields $X_k$ with the following properties: 
\begin{itemize}
\item $X_k$ is tangent to $\cF^h$
\item As $k \to \infty$, $X_k$ approaches $X$ in the $C^0$ topology
\item For any leaf $L$ of $\cF^h$, the restriction of $X_k$ to $L$ approaches the restriction of $X$ to $L$ in the $C^1$ topology\footnote{In fact, one has convergence in the $C^r$ topology for any $r$, but this is not needed.}.
\end{itemize} 
\end{lemma}

\begin{proof}
Let $D^n = (0,1)^n$ and fix a finite system of smooth local product coordinate charts $U_j  \to D^n \times (0,1) \subset \bR^{n+1}$ so that the images of local leaves of $\cF^h$ are the horizontal sets $D^n \times \{y\}$.  Let $\{ \eta_j \}$ be a partition of unity subordinate to the cover $\{U_j : j \in J\}$ indexed by a finite set. 

Fix a local chart $U_j$.  
Recall $X$, and hence $\eta_j X$ also, is tangent to $\cF^h$.  Thus, written in local coordinates in this chart, $\eta_j X$ has the form $\sum_{i=1}^{n} a^i(x,y) \frac{\partial}{ \partial x_i}$, where the functions $a^i : (D^n \times (0,1)) \to \R$ are continuous, and are smooth in the $x$-variable.   We will perturb these to become $C^1$ also in the $y$ variable, using the standard technique of convolving with a mollifier, described below. 

Let $\zeta: \mathbb{\R} \to [0,\infty)$ be a smooth bump function centered at $0$ with $\int_{\mathbb{R}} \zeta = 1$.  Thus, 
 $k \zeta(kt)$ converges, in the sense of distributions, to the Dirac delta function. 
 Writing $a = (a^1, \ldots, a^n)$, define 
\[ b_k(x, y) : = \int_{\bR} a(x,y-t) k \zeta(kt)  \, dt. \]

 One can verify that
$b_k(x,y) \to a(x,y)$ as $k \to \infty$, and it is a standard exercise to show that 
\[ \tfrac{\partial}{\partial y} b_k(x,y) = \int_{\bR} a(x,y-t) k \tfrac{\partial}{\partial t} \zeta(kt)  \, dt. \]
which is defined and continuous for all $k$.   
Finally, note that the other partial derivatives pass under the integral, so 
$ \tfrac{\partial}{\partial x_i }b_k(x, y)  = \int_{\bR} \tfrac{\partial}{\partial x_i } a(x,y-t) k \zeta(kt)  \, dt$, 
which converges to $\tfrac{\partial}{\partial x_i } a(x,y)$ as $k \to \infty$.  

Let  $X_{U_j, k}$ be the vector field on $U_j$ whose local coordinate expression is given by performing the above process to $\eta_j X$.  
Finally, we let $X_k := \sum_{j\in J} X_{U_j, k}$.  
Since $\{\eta_j\}$ is smooth and the coordinate chart system finite, the chart-wise convergence given above shows that $X_k \to X$, and that 
$X_k$ converges along leaves to $X$ in the $C^1$ sense.  
\end{proof}

\begin{lemma}  \label{lem_still_hyperbolic_on_leaves}
For $k$ sufficiently large, the restriction of the flow generated by $X_k$ to each leaf $\wt M \times \{p\}$ is forward contracting,   has orbits which are uniform (over all leaves) quasigeodesics with forward endpoint $f(p)$, and for each $\xi \neq f(p) \in \partial M$, there is a unique orbit on $\wt M \times \{p\}$ with negative endpoint $\xi$.  
\end{lemma}
In the statement, we have slightly abused notation to let $X_k$ also denote the lift of the vector field to $\wt M \times S^1$; in the proof we keep this notation use $\varphi^t$ and $\psi^t$ (see below) for the lifted flows, to reduce clutter.

\begin{proof} 
This is a standard cone-field style argument (see also  \cite[Lemma 4.4]{Asaoka}), but it will 
be restricted to leaves of $\cF^h$ and tangent vectors to such leaves.  

Since $X$ is leafwise hyperbolic, there is a cone field\footnote{As is standard, we denote  $C_\beta(E, F) := \{ u+w : u\in E, w \in F, ||w||<\beta||u||\}$, for $\beta>0$.}  $C_\beta(E^s_x, X_x)$,  around
$E^s_x$ ($E^s$ is the vector bundle tangent to horospheres
centered at $f(p)$) and $\alpha > 1$, $\lambda \in (0,1)$ so that for all $t>0$ we have: 
\begin{align*}
& \varphi^{-t}_\ast(\overline{C_\beta (E^s_x, X_x)}) \ \subset \ 
C_{\lambda^{t} \beta}(E^s_{\varphi^{-t}(x)}, X_{\varphi^{-t}(x)}) \text{ and } \\
& || \varphi^{-t}_\ast (v)|| > \alpha^t ||v|| \text{ for all } v \in C_\beta(E^s_x, X_x). 
\end{align*}
See \cite[Proposition 5.1.7]{FisherHasselblatt}.
If $X_k$ is sufficiently $C^1$ close to $X$ (along $\cF^h$ leaves), and $\psi^t$ the flow generated by $X_k$, there
are $\alpha', \lambda'$ with $1 < \alpha' < \alpha$, $\lambda < \lambda' < 1$ such that: 
\[D\psi^{-t}(\overline{C_\beta (E^s_x, X_x)}) \ \subset \ C_{(\lambda')^t \beta}(E^s_{\psi^{-t}(x)}, X_{\psi^{-t}(x)})\]
and $|| D\varphi^{-t} (v)|| > (\alpha')^t ||v||$ for $v \in C_\beta(E^s_x, X_x)$ and $t \geq 0$.
This shows that $\psi^t$ is hyperbolic, has invariant strong stable distribution given by 
\[ E^s_{\psi}(x) = \bigcap_{t>0} D\psi^{-t}(\overline{C_\beta(E^s_{\psi^{t}(x)}, X_{\psi^{t}(x)})}).\]

Since the vector field of $\psi^t$ is $C^1$ close to that of $\varphi^t$ along
leaves of $\cF^h$, it follows that the $\psi^t$ flow lines are quasigeodesics in leaves
of $\cF^h$ with very small geodesic curvature. Since they remain transverse to horocycles based at $f(p)$ they have the same forward endpoint $f(p)$ (this property in fact only requires that $X_k$ is $C^0$ close to $X$ along leaves).  Finally, the uniform backward expansion along the strong stable distribution for $X_k$ implies that distinct orbits have distinct negative endpoints.   
\end{proof}

By \cite[Lemma 4.4]{Asaoka}, 
all these approximating vector fields produce orbit equivalent flows.  Alternatively, one can produce an explicit orbit equivalence by (leafwise) straightening these quasigeodesic orbits by orthogonally projecting each orbit to the corresponding geodesic with the same endpoints. 
Going forward, we fix $k$ such that Lemma \ref{lem_still_hyperbolic_on_leaves} holds, let $Y = X_k$, and let $\psi^t$ denote the flow of $Y$.

Let $E^h \subset TN_\rho$ denote the tangent bundle of $\cF^h$ and let $U \subset TN_\rho$ denote the tangent bundle of the $S^1$ fibers.  
By Lemma \ref{lem_still_hyperbolic_on_leaves}, there is a $D\psi_t$-invariant splitting $E^h = E^s_\psi \oplus \R Y$, where $E^s_\psi$ is uniformly contracted by $\psi^t$. Denote by $\pi_U: TN_\rho \to U$ the projection given by the splitting $TN_\rho= U \oplus E^h$.

\begin{lemma}  \label{lem_contraction_for_psi}
There exist $c_1, c_2 > 0$ such that for every $v \in U$ and $t > 0$ one has
\begin{equation}\label{eq:derivativeU}
\| \pi_U \circ D\psi^t v \|_g > c_1 e^{t c_2} \|v\|_g. 
\end{equation}
\end{lemma} 

\begin{proof} 

Note first that $\pi_U \circ D\psi^t$ at a point $y \in N_\rho$ is actually the derivative of the holonomy of $\cF^h$ from $y$ to $\psi^t(y)$ measured in the fiber direction (what one would notate $D_y \mathrm{hol}^\psi_{y,t}$, adapting the notation from before).  We prove the Lemma by controlling this in terms of the derivative of holonomy of $\varphi$. 

Because orbits are uniform quasigeodesics, there exists $c>1$ and $R>0$ such that for any $y \in \wt{M} \times \{p\}$ and any $t>0$, there exists $s \in [t/c , ct]$ such that $d(\wt\varphi^s(y), \wt\psi^t(y))< R$.   Fix such $t, s$ and let $z = \varphi^{s}(y)$, so $\varphi^{-s}(z) = y$
This uniform bound implies that, there is a constant $c_3$ (which depends only on $R$, not on $t$) so that for any $u \in U_z$ we have 
\[ 
|| \pi_U \circ D\psi^t \circ D\mathrm{hol}^{\varphi}_{z, -s}(u)||_g > c_3||u||_g 
\] 
because $\pi_U \circ D\psi^t \circ D\mathrm{hol}^{\varphi}_{z, -s}$ is the infinitesemal holonomy of $\cF^h$ (measured in the $S^1$ fiber) between points distance $<R$ apart, and $N_\rho$ is compact.   
Setting $v =  D\mathrm{hol}^{\varphi}_{z, -s}(u)$ we then have 
\[ 
|| \pi_U \circ D\psi^t (v)||_g >  c_3 || D\mathrm{hol}^{\varphi}_{y, s}(v)||_g
\] 
If $k_1, k_2$ are the constants from Lemma \ref{lem_transverse_expansion}, we then have
\[ 
||\pi_U \circ D\psi^t (v) ||_g > c_3 k_1 \, e^{s k_2} ||v||_g \geq c_3 k_1 \, e^{t k_2/c} ||v||_g,
\] 
as desired. 

\end{proof}

\begin{lemma} 
The flow $\psi^t$  is Anosov.  
\end{lemma}
The proof in fact applies generally to any smooth flow satisfying the conditions and conclusions of Lemma \ref{lem_transverse_expansion}. 

\begin{proof}
By Lemma \ref{lem_still_hyperbolic_on_leaves}, there is a $D\psi^t$-invariant splitting $E^h = E^s_\psi \oplus \R Y$, where $E^s_\psi$ is uniformly contracted by $\psi^t$. 
Thus, we need only find a strong unstable distribution for $\psi^t$.  We do this using Lemma \ref{lem_contraction_for_psi} and a cone field argument.

By the uniform contraction in $E^s_\psi$ and the fact that $E^h= \bR Y \oplus E^s_\psi$, there there is some constant $c_3>0$ so that for $y \in N_\rho$, $t>0$ and $w \in E^h$ we have

\begin{equation}\label{eq:derivativeH}
\|D\psi^t w \|_g \leq c_3 \|w\|_g
\end{equation}

Fix $T\gg 1$ so that $c_1 e^{t c_2} > 2 c_3$ for all $t>T$, where $c_1, c_2$ are as in \eqref{eq:derivativeU}.  

Suppose $y \in N_\rho$ and $u \in U_y$, fix $t>0$, and write $D\psi^t u = u_t + w_t$ with $u_t \in U_{\psi^t(y)}$ and $w_t \in E^h_{\psi^t(y)}$.  Since $E^h$ is $D\psi^t$-invariant, $D\psi^t u$ stays at bounded angle from $E^h$ over bounded time, thus there exists a constant $c_4>0$ such that $\|w_t\|_g \leq c_4 \|u_t\|_g$ for all $t \in [T, 2T]$.  Moreover, by compactness of $N_\rho$ this constant can be taken uniform over all $y$ and $u \in U_y$.

Now, fix a cone angle $\beta > 2c_4$.   If $v= u+w \in T_y N_\rho$ with $u \in U_y$ and $w \in E^h_y$ and such that $\|w\|_g \leq \beta \|u\|_g$ we get that $D\psi^t v = u_t + w_t + D\psi^t w$ where $u_t \in U_{\psi^t(y)}$ and $w_t + D \psi^t w \in E^h_{\psi^t(y)}$.  Provided $t \in [T, 2T]$, by the computation above we have 
that $\|u_t\|_g > 2 c_3 \|u\|_g$ while $\|w_t\|_g < c_4 \|u_t\|_g$ and $\|D\psi^t w\|_g \leq c_3 \|w\|_g \leq c_3 \beta \|u\|_g$.  Therefore by our choice of $\beta$ we get

$$ \| w_t + D\psi^t w\|_g < c_4 \|u_t\|_g + c_3 \beta \|u\|_g < \left(c_4 + \frac{\beta}{2}\right) \|u_t\| < \beta \|u_t\|. $$

This implies that $D\psi^t( \overline{C_\beta(U,E^h)}) \subset C_\beta(U,E^h)$ for $t \in [T, 2T]$.  Iterating, this implies that the same containment holds for all $t>T$.  Thus, we can consider: 
\[ E^u_{\psi}(x) := \bigcap_{t>T} D\psi^{t}\left( \overline{C_\beta(U_{\psi^{-t}(x)}, E^h_{\psi^{-t}(x)})} \right)\]
which is a flow invariant bundle.  By \eqref{eq:derivativeU} it is uniformly expanded, and therefore the flow $\psi^t$ is Anosov using\footnote{Formally, \cite{FisherHasselblatt} ask $\beta \in (0,1)$ but this is not needed, and can always be achieved by a smooth change of metric making the circle fibers much longer.} \cite[Proposition 5.1.7]{FisherHasselblatt}.   This concludes the proof of the Lemma, and hence the smooth case of Theorem \ref{thm_main}.
\end{proof}

Note that $\psi^t$ can also be approximated by a smooth, volume preserving Anosov flow, by the main result of \cite{Asaoka}. 

\section{Cannon Thurston maps and proof of Theorem \ref{thm_counterexample}}

Given Theorem \ref{thm_main}, in order to construct explicit examples of exotic topological codimension one Anosov flows, it remains to describe minimal actions of $\pi_1(M)$ on $S^1$ (for $M$ negatively curved) with equivariant maps to $\partial \wt M$.  For smooth flows, we need additionally that the non-injective set of the map has zero measure image in $\partial \wt M$.
Many examples of these maps and actions are known to exist when $M$ is a hyperbolic 3-manifold. We show here that, for all examples coming from {\em boundaries of orbit spaces of pseudo-Anosov flows}, the noninjectivity points always form a null set.  To do this, we  first summarize the set-up and construction.   

\subsection{Pseudo-Anosov flows and actions on $S^1$}
A {\em pseudo-Anosov flow} is a flow on a compact 3-manifold, satisfying the conditions of Definition \ref{def_TAF} (topological Anosov flow), but where the foliations $\cF^s$ and $\cF^u$ are permitted to have finitely many prong-type singularities along periodic orbits.   Thus, it is a strict generalization of topological Anosov flow (in dimension 3), and we will use the term pseudo-Anosov flow to refer to both types.  
See \cite{BM} for a formal definition.  
The most basic example is the suspension flow of a pseudo-Anosov homeomorphism of a surface of genus $g \geq 2$
\cite{CT}.  In fact,  suspensions are already a very large class of examples: by Agol's virtual fibering theorem \cite{Agol}, every hyperbolic 3-manifold has a finite cover supporting such a suspension flow.  

Pseudo-Anosov flows fall into two categories, {\em $\mathbb{R}$-covered} (meaning the leaf space of the lift of $\cF^s$ to $\wt M$ is homeomorphic to $\mathbb{R}$), or not.  There are many constructions of both types in the literature going back at least 40 years.  Recently, Bowden--Mann \cite{BoM} and Beguin--Yu \cite{BY} gave constructions of hyperbolic 3-manifolds supporting arbitrarily many $\mathbb{R}$-covered, and non-$\mathbb{R}$-covered Anosov flows, respectively.  

In this section, we consider only pseudo-Anosov flows on hyperbolic 3-manifolds which are either $\mathbb{R}$-covered or {\em quasigeodesic}, the latter property meaning that the orbits lift to uniform quasigeodesics in $\wt{M}$.  For {\em Anosov} flows  a strict dichotomy is shown in \cite{FenleyQGAnosov}: any Anosov flow on a hyperbolic 3-manifold $M$ is either $\mathbb{R}$-covered or quasigeodesic.  

Following \cite{Fenley-extension}, associated to such a flow $\Phi$, one can construct a natural action of $\pi_1(M)$ on $S^1$, as specified in the following two (depending on whether $\Phi$ is $\mathbb{R}$-covered or not) constructions:

\begin{construction}[Non $\mathbb{R}$-covered case]
The {\em orbit space} of a pseudo-Anosov flow is the quotient $\cO_\Phi := \wt M/ (x \sim \wt \Phi^t(x))$ whose study was initiated by Barbot and Fenley in the 1990s.  For a pseudo-Anosov flow, $\cO_\Phi$ is always a topological plane, with invariant (singular) foliations induced from the lifted foliations $\wt \cF^s$ and $\wt \cF^u$.  In \cite{Fenley-extension} it is shown that $\cO_\Phi$ can be compactified to a disk $\overline{\cO}_\Phi$ by adding a circle $S^1_\Phi$ of ideal points.  Informally speaking, this circle is a {\em completion} of the circularly ordered set of ends of leaves of the weak-stable and weak-unstable foliations in $\cO_\Phi$, with some suitable identifications.  See \cite[\S 3]{BM} for an expository account. 
 
The action of $\pi_1(M)$ on leaves extends naturally to an action on $S^1_\Phi$ by homeomorphisms, and when $\Phi$ is not $\mathbb{R}$-covered, this action is minimal \cite[Theorem 7.3]{BBM}.  
\end{construction} 

\begin{construction}[$\mathbb{R}$-covered case]\label{const_r_covered}
When $\Phi$ is a transversally orientable $\mathbb{R}$-covered Anosov flow, $\pi_1(M)$ acts by orientation-preserving homeomorphisms on  the space of weak-stable leaves in $\wt M$, which is by definition $\bR$. 

Additionally, any $\mathbb{R}$-covered Anosov flow which is not orbit equivalent to a suspension of a linear map of the torus (in particular, all that occur on hyperbolic 3-manifolds) has a particular structure called {\em skew}, which implies that this action commutes with a translation of $\bR$ because of the orientability assumption. 
Thus, the action of $\pi_1(M)$ on the leaf space descends to the quotient by this translation $S^1_\Phi = \mathbb{R}/\bZ$, which is also minimal \cite{Fenley-Anosov}.  
\end{construction} 

Remarkably, this action can also be recovered from the non-$\mathbb{R}$ covered setting above:  When $M$ has a transversely orientable, $\bR$-covered Anosov flow $\Phi$, by \cite{Thurston} (see also \cite{Fenley-transversepAflow, Calegari-Rcov}), it also supports a quasigeodesic  pseudo-Anosov flow $\Psi$ transverse to the weak-stable leaves of $\Phi$, and the action of $\pi_1(M)$ on the boundary of $\cO_\Psi$ agrees with the above action on $S^1_\Phi=\mathbb{R}/\bZ$ (see e.g. \cite[Proposition 6.6]{Pot-Anosov} for a precise statement and a description of the identification).

So far it seems we have gained nothing new from the previous construction. However, the difference is that
in this case the action of $\pi_1(M)$ on $S^1_\Phi=\mathbb{R}/\bZ$ obviously lifts to
an action on $\mathbb{R}$ (whereas in other examples, such as from pseudo-Anosov suspensions, the actions often do not lift to an action on
$\mathbb{R}$). This extra property will be used for Corollary \ref{cor_infiniteness}. 

\subsection{Cannon-Thurston maps}\label{ss.CTmaps}
We now describe how to pass from the orbit space and action of $\pi_1(M)$ on the circle, to an equivariant map to $\partial \wt M$.  This was done for pseudo-Anosov suspensions by Cannon and Thurston \cite{CT}, and generalized to the quasigeodesic case in \cite{Fenley-ext2}. 

Let $\Phi^t$ be a quasigeodesic (pseudo-)Anosov flow on a hyperbolic $3$-manifold $M$.  
Since $\Phi^t$ is quasigeodesic, each orbit $\gamma$ of the lifted flow on $\wt M$ has a positive endpoint $e^+(\gamma)$ and negative endpoint $e^-(\gamma)$ on $\partial \wt M$.  

\begin{theorem}[ \cite{Fenley-ext2}, Theorem H] \label{thm_map_f}
Let $\Phi^t$ be a quasigeodesic (pseudo-)Anosov flow in a closed hyperbolic 
$3$-manifold $M$.
Let $S$ denote the sphere obtained by gluing two copies $\overline{\cO}_\Phi^1$ and $\overline{\cO}_\Phi^2$ of the compactified orbit space of $\Phi$ together by the identity map along their boundary $S^1_\Phi$.   
Let $e: {\cO}_\Phi^1 \cup {\cO}_\Phi^2 \to \partial \wt M$ be the map which restricts to $e^+$ on ${\cO}_\Phi^1$, and $e^-$ on ${\cO}_\Phi^2$.  

Then $e$ extends  to a continuous map $f: S \to \partial \wt M$ whose restriction to $S^1_\Phi$ is $\pi_1(M)$-equivariant. 
\end{theorem} 

Since $e^-$ is constant along leaves of $\wt \cF^u$, and $e^+$ is constant along leaves of $\wt \cF^s$ (due to the asymptotic behavior of orbits) the map $f$ can alternatively be described as a quotient:  Consider the decomposition of $S$ where decomposition elements are compactified stable leaves in $\overline{\cO}_\Phi^1$, compactified 
unstable leaves in $\overline{\cO}_\Phi^2$, and otherwise points (necessarily in $S^1_\Phi$; these are the ``completion points" in the boundary circle which do not arise as endpoints of leaves). This decomposition is upper semi-continuous (see  \cite[Section 6]{Fenley-ext2}), so the classical decomposition theorem of Moore applies (see \cite{Fenley-ext2}), and the quotient is a topological sphere.  This sphere can then naturally be identified with $\partial \wt M$ by using the dynamics of the flow.  
Though not explicitly stated in \cite{Fenley-ext2}, the construction there, as we have also described above, directly gives the following: 
\begin{proposition}
The set of non-injectivity points of $f$ is contained in the image of $e$. 
\end{proposition}

\subsection{Conclusion of proof of Theorem \ref{thm_counterexample}}

Given the discussion above, to produce examples of exotic codimension one Anosov flows, and prove Corollary \ref{coro-4manifold}, it suffices to show the following: 
\begin{proposition}  \label{prop_null}
Let $\Phi^t$ be a quasigeodesic (pseudo-)Anosov flow on a hyperbolic $3$-manifold $M$, and $f$ as in Theorem 
\ref{thm_map_f}.  Then the set of points in $\partial \wt M$ with multiple pre-images under $f$ has\footnote{Using \cite{LY} one could actually show that this set has Hausdorff dimension less than 2, but we don't need this stronger fact.} measure 0. 
\end{proposition} 

For the proof, we first give a general lemma: 

\begin{lemma} \label{lem_positive_implies_everything}
Let $M$ be a closed, hyperbolic manifold and $\Lambda \subset T^1M$ a closed subset invariant under geodesic flow.  If there exists a strong stable leaf $W^{ss}$ for geodesic flow such that $\Lambda \cap W^{ss}$ has positive Lebesgue measure in $W^{ss}$, then $\Lambda = T^1M$.  
\end{lemma}

\begin{proof} 
We use the Sasaki metric in $T^1 M$, which has the property that flowing backwards from the strong stable manifold of $v$ to the strong stable manifold $g_{-t}(v)$ by the geodesic flow rescales the metric by a factor of $e^t$. 

Assume that the Lebesgue measure of some strong stable manifold of the geodesic flow $g_t: T^1 M \to T^1 M$ intersected with $\Lambda$ is positive.  
Let $v \in  \Lambda \cap W^{ss}_{loc}(v)$  be a Lebesgue density point of $\Lambda \cap W^{ss}_{loc}(v)$ where $W^{ss}_{loc}(v)$ denotes the local strong stable manifold of the geodesic flow at the point $v$. This means that there are disks  $D_n \subset W^{ss}_{loc}(v)$ centred at $v$ and of radius $\leq 1/n$ so that $\frac{m_{ss}(D_n \cap  \Lambda)}{m_{ss}(D_n)} \to 1$ where $m_{ss}$ denotes the Lebesgue measure on strong stable manifolds. Flowing backward by the geodesic flow we get that for for some times $t_n \to \infty$ we obtain disks $\hat D_n \subset W^{ss}_{loc}(g_{-t_n}(v))$ of some fixed radius $\eps>0$ around $g_{-t_n}(v)$. By the elementary fact above and invariance of $\Lambda$ we get that $\frac{m_{ss}(\hat D_n \cap \Lambda)}{m_{ss}(\hat D_n)} \to 1$.  By compactness of $M$, after a subsequence we can assume that $\hat D_n \to D_\infty$ a closed disk of radius $\eps$ in some strong stable manifold. By the density computation, we get that $\hat D_n \cap  \Lambda \to D_\infty$ in the Hausdorff topology, and therefore $D_\infty \subset \Lambda$, since $\Lambda$ is closed. 
The set of periodic orbits of the geodesic flow is dense, so some periodic orbit intersects $D_\infty$ in the interior, and as $\Lambda$ is flow invariant, it follows that a strong stable leaf is contained in $\Lambda$. Since the strong stable foliation is minimal in $T^1M$, and $\Lambda$ is closed, this implies that $\Lambda = T^1 M$.   
\end{proof}

\begin{proof}[Proof of Proposition \ref{prop_null}]
As described above, the orbit space $\cO_\Phi$ inherits a pair of 1-dimensional, topologically transverse, singular foliations induced from $\wt \cF^s$ and $\wt \cF^u$, with isolated prong singularities.  
 Thus,  $\cO_\Phi$ can be covered by countably many compact rectangles, each trivially foliated as a product.  
We will show that, for any such box $B$, the images  $e^+(B)$ and $e^-(B)$ have 
Lebesgue measure zero. 

Consider the map $(e^-, e^+): \cO_\Phi \to \partial \wt{M} \times \partial \wt{M}$  given by $\gamma \mapsto (e^-(\gamma), e^+(\gamma))$.  This map $(e^-, e^+)$ is continuous and $\pi_1(M)$-equivariant.  Since orbits of quasigeodesic flows are uniformly quasigeodesic and have distinct forward and backwards endpoints (when lifted to the universal cover), the image avoids the diagonal hence is a $\pi_1(M)$-invariant closed subset of 
$\wt \Lambda \subset \partial \wt{M} \times \partial  \wt{M} \setminus \Delta$, which is the set of oriented geodesics in $T^1\wt{M}$.  Thus, $\wt \Lambda$ descends to a subset $\Lambda$ of $(T^1\wt M)/\pi_1(M) = T^1M$ that is invariant under the geodesic flow (see e.g \cite[\S 10]{FrankelLandry} for more details on this construction).  

Since $\wt \Lambda$ is closed, $\Lambda$ is compact.  Furthermore, $\Lambda \neq T^1 M$ because any pseudo-Anosov flow $\Phi$ on a $3$-manifold $M$ has the property that not every free homotopy class of closed, oriented geodesic in $M$ (that is, a periodic orbit of geodesic flow) is realized by an orbit of $\Phi$.  (See for instance \cite[Proposition 2.9.3]{BM} which says for any pseudo-Anosov flow, there are many $\gamma \in \pi_1(M)$ so that at least one of $\gamma$ or $\gamma^{-1}$ is not freely homotopic to a periodic orbit.) 

By Lemma \ref{lem_positive_implies_everything}, this means that $\Lambda \cap W^{ss}$ has zero Lebesgue measure in $W^{ss}$, for every strong stable leaf.   
Consider a trivially foliated rectangle $B$ as above.  Identify $B$ with $[0,1]^2$ where the horizontal foliation is the stable foliation in $\cO_\Phi$, and the vertical foliation is unstable in $\cO_\Phi$.   Since $e^+$ (resp. $e^-$) is constant along stable (resp. unstable) leaves, the image of $B$ under $e^-$ is equal to $e^- ([0,1] \times \{0\})$.  
Now the map $(e^-, e^+)$ sends $[0,1] \times \{0\})$ into a single stable leaf of $T^1 \wt M$.  
By Lemma \ref{lem_positive_implies_everything}, this image intersects the stable leaf in a Lebesgue null set.  Thus, the corresponding negative endpoints (on $\partial \wt M$) is also a null set, since projection from strong stable leaves along the flow to the boundary in hyperbolic space is smooth.  
The argument for $e^-$ is exactly the same.  
\end{proof} 

\subsection{Infiniteness: end of proof of Theorem \ref{thm_counterexample} and Corollary \ref{cor_infiniteness} }
The flexibility in the construction of $f$ allows for infinitely many examples, even up to covers:   Since there are infinitely many commensurability classes of hyperbolic 3-manifolds (see e.g. \cite[Corollary 8.4.2]{MR}),
and virtual fibering \cite{Agol} ensures that each of these contains a 3-manifold with a quasigeodesic pseudo-Anosov flow, we obtain an infinite sequence of incommensurable $M_i$ such that some circle bundle $N_i$ over $M_i$ admits a codimension one Anosov flow; and these $N_i$ are pairwise incommensurable as well.  

The final assertion of Theorem \ref{thm_counterexample} is given by the failure of finiteness statement in Corollary \ref{cor_infiniteness}.  To prove this, consider any orientable $\mathbb{R}$-covered Anosov flow on a hyperbolic 3-manifold, and 
perform Construction \ref{const_r_covered} to get an action $\rho$ of $\pi_1(M)$ on $S^1 = \mathbb{R}/\mathbb{Z}$  (which agrees with the boundary of the orbit space action for some transverse quasigeodesic pseudo-Anosov flow) as well as a map $f: S^1 \to \partial \wt{M}$ via the identification of the circles explained there.  Since this is induced from an action of $\pi_1(M)$ on $\mathbb{R}$, it lifts to any finite cover $\mathbb{R}/k \mathbb{Z}$.  Let $\rho_k$ denote this lift.   This gives us an infinite family of manifolds $N_{\rho_k}$, with Anosov flows $\psi_k$, each pair with a common a fiberwise cover.

Since $\rho$ lifts to an action $\hat \rho$ on $\mathbb{R}$, each $N_{\rho_k}$ is diffeomorphic to $M \times S^1$.  This can be seen by Euler number considerations, but also directly since it is 
diffeomorphic to the quotient of $\wt M \times \mathbb{R}$ by the diagonal action of $\pi_1$ via $\hat \rho$ and a (central) map which is identity on the $\wt M$ factor and translation on the $\mathbb{R}$ factor.   
However, the flows $\psi_k$ are not orbit-equivalent:  to see this, let $m$ denote the minimum number of closed orbits in any fixed free homotopy class for $\psi_1$.  Then $\psi_k$ has, by construction, a minimum number $mk$.  

In the same way, if $M$ admits a non-$\mathbb{R}$-covered Anosov flow, it follows by \cite{FenleyQGAnosov} that it is quasigeodesic and the discussion in \S~\ref{ss.CTmaps} applies. In addition the circle at infinity of $\Phi$ is a universal circle
for the stable foliation of $\Phi$, see \cite[\S 4.3]{BT}. Using \cite[Proposition 7.1]{BoyerHu} we know that the Euler class of $\rho$ is zero, and therefore the resulting circle bundle $N_\rho = M \times S^1$ and we can do the same argument as above. 

This concludes the proof of Corollary \ref{cor_infiniteness} and Theorem \ref{thm_counterexample}.



\begin{thebibliography}{2}

\bibitem[Ago]{Agol} I. Agol, The virtual haken conjecture, Doc. Math. {\bf 18} (2013), 1045--1087.

\bibitem[Asa]{Asaoka} M. Asaoka, On invariant volumes of codimension-one Anosov flows and the Verjovsky conjecture, Invent. math. {\bf 174}, (2008) 435--462. \emph{Erratum Invent. Math. {\bf 178} (2009), no. 2, 449.}

\bibitem[Bar]{Barbot} T. Barbot, G\'eom\'etrie transverse des flots d'Anosov, Ph.D. thesis, Ecole Normale Sup\'erieure de Lyon, 1992.

\bibitem[BBGH]{BBGH} T. Barthelme, C. Bonatti, A. Gogolev, F. Rodriguez Hertz,  Anomalous Anosov flows revisited,  Proc. Lond. Math. Soc. (3) {\bf 122} (2021), no. 1, 93--117.

\bibitem[BBM]{BBM} T. Barthelme, C. Bonatti, K. Mann, Non-transitive pseudo-Anosov flows, arXiv: 2411.03586.


\bibitem[BFM]{BFM} T. Barthelme, S. Frankel, K. Mann, Orbit equivalences of pseudo-Anosov flows, Invent. Math.  {\bf 240} (2025), no. 3, 1119--1192.

\bibitem[BM]{BM} T. Barthelme, K. Mann, Pseudo Anosov flows, a plane approach, arXiv:2509.15375 

\bibitem[BY]{BY} F. Beguin, B. Yu, Existence of arbitrary large numbers of non-$\R$-covered Anosov flows on hyperbolic  3-manifolds, arXiv:2402.06551


\bibitem[BHLM]{BHLM} I. Bhattacharyya, R. Halder, N. Lazarovich, M. Mj, Finiteness of Cannon--Thurston fibers, arXiv:2603.22428

\bibitem[BoI]{BI-flows} C. Bonatti, I. Iakovoglou, Anosov flows on 3-manifolds: the surgeries 
	and the foliations, Erg. Th. Dyn. Sys. {\bf 43} (2023) 1129-1188.

\bibitem[BoM]{BoM} J. Bowden, K. Mann, $C^0$ stability of boundary actions and inequivalent Anosov flows, Annales Scientifique de l'ENS. 4e serie, {\bf 55} (2022) 1003-1046. 

\bibitem[BoHu]{BoyerHu} S. Boyer, Y. Hu, Taut foliations in branched cyclic covers and left-orderable groups, Transactions of the AMS, {\bf 372} 11 (2019) 7921--7957. 


\bibitem[BT]{BT} E. Buckminster, S. Taylor, Universal circles for Anosov foliations, arXiv: 2512.10107.


\bibitem[Ca$_1$]{Calegari-Rcov} D. Calegari, The geometry of R-covered foliations, Geometry and Topology {\bf 4} (2000) 457-515.

\bibitem[Ca$_2$]{Calegari} D. Calegari, Promoting essential laminations, Invent. Math. {\bf 166} (2006), no. 3, 583--643.

\bibitem[CL]{CL} D. Calegari, I. Loukidou , CaTherine Wheels, arXiv:2604.24619


\bibitem[CT]{CT} J. Cannon; W. Thurston, Group invariant Peano curves,  Geometry and Topology. {\bf 11} (3) (2007) 1315--1356.

\bibitem[Caw]{Caw} E. Cawley, The Teichm\"{u}ller space of an Anosov diffeomorphism of $T^2$, Inventiones Math. {\bf 112} (1993) 351--376. 


\bibitem[Fen$_1$]{Fenley-Anosov} S. Fenley, Anosov flows in 3-manifolds, Ann. Math. {\bf 139} (1994) 79-115.

\bibitem[Fen$_2$]{Fenley-transversepAflow} S. Fenley, Foliations, topology and geometry of 3-manifolds: R-covered foliations and transverse pseudo-Anosov flows, Comm. Math. Helv. {\bf 77} (2002) 415-490.


\bibitem[Fen$_3$]{Fenley-extension} S. Fenley, Ideal boundaries of pseudo-Anosov flows and uniform convergence groups, with connections and applications to large scale geometry, Geom. Topol. {\bf 16} (2012), no. 1, 1--110.

\bibitem[Fen$_4$]{Fenley-ext2} S. Fenley, Quasigeodesic pseudo-Anosov flows in hyperbolic 3-manifolds and connections with large scale geometry,Adv. Math. {\bf 303} (2016), 192--278.

\bibitem[Fen$_5$]{FenleyQGAnosov} S. Fenley, Non R-covered Anosov flows in hyperbolic 3-manifolds are quasigeodesic, Geom. Funct. Anal. {\bf 36} (2026) no. 2, 412--508.

\bibitem[FiHa]{FisherHasselblatt} T. Fisher, B. Hasselblatt, Hyperbolic flows, Zurich Lectures in Advanced Mathematics, EMS (2019). 

\bibitem[Fr]{Franks} J. Franks, Anosov diffeomorphisms, Global Analysis (Proc. Sympos. Pure Math., vol.XIV,
Berkeley, CA,1968), American Mathematical Society, Providence,RI,1970, 61--93.

\bibitem[FL]{FrankelLandry} S. Frankel, M. Landry, From quasigeodesic to pseudo-Anosov flows, 
arXiv: 2510.02217.

\bibitem[Gh]{Ghys} E. Ghys, Codimension on Anosov flows and suspensions, Lecture Notes in Math., 1331
Springer-Verlag, Berlin, 1988, 59--72.
ISBN: 3-540-50016-2


\bibitem[LY]{LY} F. Ledrappier, L-S. Young, The Metric Entropy of Diffeomorphisms: Part II: Relations between Entropy, Exponents and Dimension, Annals of Mathematics {\bf 122} No. 3 (Nov., 1985), 540--574.

\bibitem[MR]{MR} C. Maclachlan, A. Reid, The arithmetic of hyperbolic 3-manifolds, Springer 2003.

\bibitem[Ne]{Newhouse} S. Newhouse, On codimension one Anosov diffeomorphisms, Amer. J. Math. {\bf 92}
(1970),761--770. 

\bibitem[PT]{PlanteThurston} J. Plante, W. Thurston, Anosov flows and the fundamental group, Topology,
Volume 11, Issue 2,1972, 147--150.

\bibitem[Pot]{Pot-Anosov} R. Potrie, Anosov flows in dimension 3: an outside look, J. Fixed Point Theory Appl. 27 (2025), no. 1, Paper No. 21. 

\bibitem[PS]{PujSam} E. Pujals, M. Sambarino, Density of hyperbolicity and tangencies in sectional dissipative regions, Annales de l'Institut Henri Poincar\'e. C, Analyse non lin\'eaire, Volume 26 (2009) no. 5, 1971--2000.

\bibitem[Qu]{Qu} J.-F. Quint, An overview of Patterson-Sullivan theory,	
	https://www.math.u-bordeaux.fr/~jquint/publications/courszurich.pdf.

\bibitem[Si]{Simic} S.Simic, Codimension one Anosov flows and a conjecture of Verjovsky, Ergodic Theory Dynam. Systems {\bf 17} (1997), no. 5, 1211--1231.

\bibitem[Sm]{Smale} S. Smale, Differentiable dynamical systems, Bull. Amer. Math. Soc. {\bf 73} (1967), 747--817.

\bibitem[Su]{Su} D. Sullivan, The density at infinity of a discrete group of hyperbolic motions,
	Publ. I.H.E.S. {\bf 50} (1979) 171-202.

\bibitem[Sum]{Sumi} N. Sumi, Topological Anosov maps of infra-nilmanifolds. Journal of the Mathematical Society of Japan, {\bf 48} (4) (1996) 607--648.

\bibitem[Tho]{Tholozan} N. Tholozan, ''Exotic maximal sufrace group representations into $\mathrm{Diff}(S^1)$'' Talk at IHP, available in CarminTV: https://www.carmin.tv/en/video/exotic-maximal-surface-group-representations-into-diff-s1

\bibitem[Tho$_2$]{Tholozan-inprep} N. Tholozan, \emph{In preparation}. 

\bibitem[Thu]{Thurston}  W. Thurston, Three manifolds, foliations and circles I arXiv:9712268

\bibitem[Ve$_1$]{Verjovsky} A. Verjovsky, Flows with cross sections, Proc. Nat. Acad. Sci. {\bf 66} 4 (1970) 1154--1556.

\bibitem[Ve$_2$]{Verjovsky1} A. Verjovsky, Codimension one Anosov flows, Bol. Soc. Matematica Mexicana {\bf 19} (1974), 49--77.

\bibitem[Ve$_3$]{Verjovsky2} A. Verjovsky, Sistemas de Anosov, Escuela Latinoamericana de Matematicas (XII-ELAM), Lima (1999)

\end{thebibliography}
\end{document}